\documentclass[a4paper,11pt]{article}
\usepackage[dvips]{graphicx}
\usepackage[dvips]{color}
\usepackage{subfigure}
\usepackage[english]{babel}
\usepackage{amsmath,amsfonts,amsthm,amssymb}

 \newtheorem{thm}{Theorem}[section]
 
 \newtheorem{lem}[thm]{Lemma}
 
  \newtheorem{prop}[thm]{Proposition}

\theoremstyle{definition}
\newtheorem{defn}[thm]{Definition}

 \theoremstyle{remark}
 \newtheorem{rem}[thm]{Remark}

\newcommand{\SM}{\mathcal{S}_{M}}
\newcommand{\Ho}{\mathcal{H}}
\newcommand{\La}{\Lambda}

\DeclareMathSymbol{\leqslant}    {\mathrel}{AMSa}{"36}
\DeclareMathSymbol{\geqslant}    {\mathrel}{AMSa}{"3E}

\title{New Pseudodistances Associated with Reparametrization Invariant Seminorms}

\author{M.~Montserrat Alonso Ferrero}

\begin{document}

\maketitle

\begin{abstract}

Let us consider two compact connected and locally connected Hausdorff spaces $M$, $N$ and two continuous functions $\varphi:M\to \mathbb{R}$, $\psi:N\to \mathbb{R}$ . In this paper we introduce new pseudodistances between pairs $(M,\varphi)$ and $(N,\psi)$ associated with reparametrization invariant seminorms. We study the pseudodistance associated with the seminorm $\|\varphi\|=\max\varphi-\min\varphi$, denoted by $\delta _{\Lambda}$, and we find a sharp lower bound for it. We finish with an example where the use of this lower bound is  illustrated.

\end{abstract}

\section{Introduction}

Topological Persistence is attracting increasing  attention from the mathematical community (cf. \cite{BiFlFa}, \cite{EH} and \cite{G}). It studies the properties of a scalar function $\varphi$ (defined on a topological space $X$) which are invariant with respect to perturbations. These functions are central in many applications such as shape matching. In other words, Topological Persistence measures the \lq \lq persistence \rq \rq of topological structures (e.g. connected components) within the sublevel sets, $\{x\in X: \varphi(x) < c\}$, of a scalar field.\\
A key tool in topological persistence is the concept of natural pseudodistance (cf. \cite{DFr2} and  \cite{DFr3}). It is based on a quantitative comparison between suitable topological spaces, endowed with real-valued functions. We work with pairs $(M, \varphi)$ and $(N,\psi)$ where $M$ and $N$ are compact connected and locally connected Hausdorff spaces  which represent the shapes to be compared and  $\varphi:M\to \mathbb{R}$, $\psi:N\to \mathbb{R}$ are continuous functions which focus the shape properties we are interested in. So, our aim is \lq\lq to measure\rq\rq the difference between the pairs $(M,\varphi)$, $(N,\psi)$. The concept of natural pseudodistance  is based on the search for a homeomorphism $h:M\to N$ minimizing the change from the function $\varphi$ to the function $\psi$, that is
 $$\delta((M,\varphi),(N,\psi))= \underset{h\in \Ho(M,N)}{\inf}\max _{M}\vert(\varphi - \psi \circ h)\vert$$ where $\Ho(M,N)$ denotes the set of all homeomorphisms  between the topological spaces $M$ and $N$.

In case where the topological spaces are manifolds, the study of this concept has pointed out some interesting properties and shown that these properties could depend on the dimension of the manifolds we are considering (cf. \cite{DFr3}). This approach is also interesting for application purposes (cf. \cite{BiFlFa}) and is strictly related to Persistent Homology, another field of research that concerns the comparison of manifolds endowed with real-valued functions (cf. \cite{EH}).

Another fundamental tool in Topological Persistence is the size function associated to a pair $(M,\varphi)$. Size functions not only give us a lower bound for the natural pseudodistance, but they
are also useful in many applications for Pattern Recognition (cf. \cite{FeLoP}, \cite{UV1}, \cite{UV2} and \cite{FeFrUV}).

\vskip .4cm

While the definition of natural pseudodistance is based on the $L_\infty$ norm,  in \cite{FrLa2} it has been recently pointed out that this norm can be replaced by many other (semi)norms, $\mathcal{S}_M$, so producing different pseudodistances

$$\delta_{exo}((M,\varphi),(N,\psi))=\underset{h\in \Ho(M,N)}{\inf}\mathcal{S}_M(\varphi-\psi\circ h).$$
\noindent The only requirement is that these seminorms, $\mathcal{S}_M$, must be reparametrization invariant, i.e. invariant under composition with homeomorphisms of the considered topological spaces. The new pseudodistances we obtain when we change the seminorm will be called \lq\lq exotic pseudodistances\rq\rq to differ from the pseudodistance in the  \lq\lq natural\rq\rq case. This change of the seminorm is useful in certain cases where the natural pseudodistance does not take into account some details. For example, let us  consider $M=[0,\pi]$ and $\mathcal{S}_M(\varphi)=\max \varphi-\min \varphi$. While the value taken by  the exotic pseudodistance between the pairs $(M, \sin 2t)$ and $(M, 0)$  is $2$ and the one between $(M, \sin t)$ and $(M, 0)$  is $1$, the natural pseudodistance between the same pairs  is $1$, in both cases. Following the approach exposed in \cite{FrLa2}, after the $L_\infty$ norm, the simplest example of reparametrization invariant seminorm is the seminorm $\|\xi\|=\max\xi-\min\xi$. The contribution of this paper is to give a formal introduction of the pseudodistances associated with reparametrization invariant seminorms and to begin the study of the pseudodistance $\delta_\Lambda$ associated with the seminorm $\|\xi\|=\max\xi-\min\xi$. Our main result is a sharp lower bound for $\delta_\Lambda$, obtained by using a previous lower bound for the natural pseudodistance.

The structure of this paper is as follows.
In Section 2 we recall the theoretical background, while in Section 3 the definition of exotic pseudodistance is introduced. In Section 4 our lower bound for $\delta_\Lambda$ is proved. We finish with an example (Section 5) where all these notions are illustrated.

\section{Theoretical Background}

In this section we will review the basic notions we shall require for the sections to follow. For more details about them we refer to \cite{DFr2}, \cite{DFr3} and \cite{Fr}.

We consider pairs of the form $(M,\varphi)$ where $M$ is a non-empty compact connected and locally connected Hausdorff space and $\varphi : M\longrightarrow \mathbb R$ is a continuous function called a \emph{measuring function}. The collection of these pairs will be denoted by $Size$ and each element $(M,\varphi)$ of $Size$ will be called a $\emph{size pair}$.

Let $\mathcal{C}(M,\mathbb{R})$ be the set of real valued continuous functions on $M$ and $\mathcal{H}(M,N)$ be the (possibly empty) set of all homeomorphisms  between the topological spaces $M$ and $N$. We denote by $\mathbb{R}^{+}$ the set of real numbers greater than or equal to zero.

All the following definitions can easily be extended to cases where the topological spaces underlying the size pairs are not connected.

\subsection{Natural pseudodistance and Size Functions}

For every compact topological space $M$ we consider the functional $\Theta_{M} :\mathcal{C}(M,\mathbb{R})\longrightarrow \mathbb{R}^+$ defined by setting
$$\Theta_{M} (\xi)=\max _{p\in M}\vert \xi(p)\vert.$$

\begin{rem}
 Let $(M,\varphi)$ and $(N,\psi)$ be two size pairs with $\Ho (M,N)\neq \emptyset$. The number
$$\Theta_{M} (\varphi - \psi \circ f)=\max _{p\in M}\vert\varphi (p)-\psi\circ f (p)\vert$$  measures how much $f\in \mathcal{H}(M,N)$ changes the values taken by the measuring function.
\end{rem}

\begin{defn}
 We define the \emph{natural pseudodistance} between  $(M,\varphi)$ and $(N,\psi)$ as
$$
\delta((M,\varphi),(N,\psi))=
\begin{cases}
\underset{f\in \Ho(M,N)}{\inf}\Theta_{M} (\varphi - \psi \circ f)& \textrm{ if } \mathcal{H}(M,N)\not=\emptyset,\\
\infty & \textrm{ otherwise}.
\end{cases}
$$

\end{defn}

We can verify immediately that $\delta$ is a pseudodistance in each
equivalence class of $Size/\hspace{-0.125cm}\approx$ where we say
that two size pairs are equivalent with respect to the relation
$\approx$ if and only if the underlying topological
spaces are homeomorphic. Note that $\delta$ is not a distance, since
two different size pairs  can have vanishing distance. Furthermore,
$\delta$ is an extended pseudodistance in \emph{Size}, where the
adjective extended means that $\delta$ can take the value $\infty$.

\vskip .5cm

Before giving the notion of  size function, we fix some notation. The symbol $\bigtriangleup$ denotes the diagonal of $\mathbb{R}^2$, that is
$$\bigtriangleup =\{(x,y)\in \mathbb{R}^2~:~x=y\}$$
and $\bigtriangleup ^+$ denotes the open half-plane above $\bigtriangleup$, that is
$$\bigtriangleup ^+ =\{(x,y)\in \mathbb{R}^2~:~x<y\}.$$

\begin{defn}

For each size pair $(M,\varphi)$ and $y\in \mathbb{R}$, we say that two points $p,q\in M$ are \linebreak $\langle\varphi \leqslant y\rangle$-\emph{connected} if and only if they belong to the same connected component of the lower level set $\{\overline{p}\in M~:~\varphi (\overline{p})\leqslant y\}$. Then we define $\ell  _{(M,\varphi)}(x,y)$ for $(x,y)\in\bigtriangleup ^{+}$ as the number of equivalence classes into which the lower level set $\{\overline{p}\in M~:~\varphi (\overline{p})\leqslant x\}$ is divided by the equivalence relation of $\langle\varphi \leqslant y\rangle$-connectedness. We call

$$\ell _{(M,\varphi)}: \bigtriangleup ^{+}\longrightarrow \mathbb{N}$$
the \emph{size function associated with the size pair $(M,\varphi)$}.
\end{defn}






\subsection{Matching Distance for Size Functions} \label{seccion 2.2}

In order to define the matching distance, we begin by describing how  each size function can be represented as a set of points and lines in the real plane, with multiplicities. We start with the definition of  particular points of $\mathbb{R}\times (\mathbb{R}\cup \left\{\infty\right\})$, the cornerpoints associated with a  size function (cf. \cite{FrLa1}).
\vskip .2cm
\noindent Let $(M, \varphi)$ be a size pair.
\vskip .2cm

\begin{defn}
For every point $(x,y)\in \bigtriangleup ^{+}$ and $\varepsilon \in \mathbb{R}$ with $\varepsilon>0$ and $x+\varepsilon < y-\varepsilon$ we define the number $\mu _{\varepsilon}(x,y)$ as
$$\ell _{(M,\varphi)}(x+\varepsilon, y-\varepsilon)-\ell _{(M,\varphi)}(x-\varepsilon, y-\varepsilon)-\ell _{(M,\varphi)}(x+\varepsilon, y+\varepsilon)+\ell _{(M,\varphi)}(x-\varepsilon, y+\varepsilon).$$
The finite number
$$\mu (x,y)=\min \{\mu _{\varepsilon}(x,y)\mid \varepsilon \in \mathbb{R},\varepsilon>0\mbox{ and }x+\varepsilon < y-\varepsilon\}$$
will be called the \emph{multiplicity} of $(x,y)$ for $\ell _{(M,\varphi)}$. We call \emph{cornerpoint} for $\ell _{(M,\varphi)}$ any point $(x,y)\in \bigtriangleup ^{+}$ such that the number $\mu (x,y)$ is strictly positive.
\end{defn}

\begin{defn}
For every vertical line $r$, with equation $x=k$, and  $\varepsilon \in \mathbb{R}$ with $\varepsilon>0$ and $k+\varepsilon < \dfrac{1}{\varepsilon}$ we define the number $\mu_{\varepsilon}(r)$ as
$$\ell _{(M,\varphi)}\left(k+\varepsilon, \dfrac{1}{\varepsilon}\right)-\ell _{(M,\varphi)}\left(k-\varepsilon, \dfrac{1}{\varepsilon}\right).$$
The finite number
$$\mu (r)=\min \left\{\mu_{\varepsilon}(r)\mid \varepsilon \in \mathbb{R}, \varepsilon>0\mbox{ and   }k+\varepsilon <\dfrac{1}{\varepsilon}\right\} $$
will be called the \emph{multiplicity} of $r$ for $\ell _{(M,\varphi)}$. When this multiplicity is strictly positive we call $r$ a \emph{cornerpoint at infinity} for the  size function, and we identify $r$ with the pair $(k,\infty)$.

\end{defn}

We denote by $\bigtriangleup ^{*}$  the open  half-plane $\bigtriangleup ^{+}$ extended by the points at infinity of the kind $(k,\infty)$, with $\vert k\vert<+\infty$, and by $\overline{\bigtriangleup ^{*}}$ the closed half-plane $\overline{\bigtriangleup^{+}} $ extended by the same points at infinity.

Cornerpoints and their multiplicities are fundamental tools in Size Theory, since size functions are  completely determined by their cornerpoints and multiplicities. Indeed, we have the following representation theorem (cf. \cite{FrLa1}).
\begin{thm}
 For every $(x',y')\in \bigtriangleup ^{+}$ we have
$$\ell _{(M,\varphi)}(x',y')=\sum_{\substack{(x,y)\in \bigtriangleup ^{*}\\ x\leq x', y'<y}}\mu((x,y)).$$

\end{thm}

We recall some definitions.
\begin{defn}
Let $\ell $ be a  size function. We call \emph{representative sequence} for $\ell$ any sequence of points $(a_{n})_{n\in \mathbb{N}}$ in $\overline{\bigtriangleup^{*}}$ with the following properties:
\begin{itemize}
\item [i)] $a_{0}$ is the cornerpoint at infinity for $\ell $;
\item[ii)] for each $n>0$, either $a_{n}$ is a proper cornerpoint for $\ell $, or $a_{n}$ belongs to $\bigtriangleup$;
\item[iii)] if $(x,y)$ is a proper cornerpoint for $\ell $ with multiplicity $\mu (x,y)$, then the cardinality of the set $\{n\in \mathbb{N}:a_{n}=(x,y)\}$ is equal to $\mu (x,y)$;
\item[iv)] the set of indexes for which $a_{n}$ belongs to $\bigtriangleup$ is countably infinite.
\end{itemize}

\end{defn}

Now, we define a pseudometric in $\overline{\bigtriangleup^{*}}$ that will give rise to a distance between  size functions.

\begin{defn}
We define the function  $d:\overline{\bigtriangleup^{*}} \times \overline{\bigtriangleup^{*}}\longrightarrow \mathbb{R}^{+}$  by

$$d((x,y),(x',y'))=\min \left\{\max \{\vert x-x'\vert,\vert y-y'\vert\},\max\left\{\dfrac{y-x}{2},\dfrac{y'-x'}{2}\right\}\right\}$$

\noindent where conventions regarding $\infty$ are: $\infty -y=y-\infty=\infty$ for $y\neq \infty$, $\infty -\infty=0$, $\dfrac{\infty}{2}=\infty$, $\vert \infty\vert=\infty$, $\min\{\infty,c\}=c$, $\max \{\infty, c\}=\infty$.
\end{defn}

The function $d$ is a pseudodistance on $\overline{\bigtriangleup^{*}}$, it measures the smaller between the cost of moving $(x,y)$ to $(x',y')$, and the cost of moving $(x,y)$ and $(x',y')$ onto the diagonal. The costs are computed by using the distance induced by the max-norm.

\begin{defn}
 Let $(a_{n})$ and $(b_{n})$ be two representative sequences for two size functions $\ell _{1}$ and $\ell _{2}$, respectively. The \emph{matching distance} between $\ell _{1}$ and $\ell _{2}$ is the number
$$d_{match}(\ell _{1},\ell _{2}):=\inf _{\sigma}\sup _{n} d(a_{n},b_{\sigma (n)}),$$
where $n$ varies in $\mathbb{N}$ and $\sigma$ varies among all the bijections from $\mathbb{N}$ to $\mathbb{N}$.
\end{defn}

\begin{rem}
It is easy to see that this definition is independent from the choice of the representative sequences of points for the  size functions $\ell _{1}$ and $\ell _{2}$. The $\inf$ and the $\sup$ in the definition of matching distance are actually attained, that is
$$d_{match}(\ell _{1},\ell _{2}):=\min _{\sigma}\max _{n} d(a_{n},b_{\sigma (n)}).$$
\end{rem}

Now we recall an important result concerning a  lower bound for the natural pseudodistance using the matching distance (cf. \cite{dAFrLa}).

\begin{thm}\label{desigualnatural}
 Let $(M,\varphi)$ and $(N,\psi)$ be two size pairs. Then
$$d_{match}(\ell _{(M,\varphi)}\ell _{(N,\psi)})\le\delta ((M,\varphi),(N,\psi)).$$
\end{thm}

\section{Exotic Pseudodistance}

In this section we introduce a new family of pseudodistances associated with other seminorms. The fundamental property of the seminorms that we take into account is an invariance property for homeomorphisms. We will see that in certain cases where the natural pseudosistance does not distinguish two size pairs, these new pseudodistances allow us to compare them in a sharper way, so better quantifying their  differences (see example in Section \ref{example}).

\vskip .3cm

Let us assume that for every compact topological space $M$  a functional
$\mathcal{S}_{M}: \mathcal{C}(M,\mathbb{R})\longrightarrow
\mathbb{R}^+$ is given verifying the following  properties:
\begin{description}
\item [i)] $\mathcal{S}_{M}(\varphi)\geqslant 0$ for all $\varphi \in \mathcal{C}(M,\mathbb{R})$.
 \item[ii)] $\mathcal{S}_{M}(\lambda \varphi)=\vert\lambda\vert\cdot \mathcal{S}_{M}(\varphi)$ for all $\lambda \in \mathbb{R}$ and $\varphi \in \mathcal{C}(M,\mathbb{R})$.
\item[iii)] $\mathcal{S}_{M}(\varphi _{1}+\varphi _{2})\leq \mathcal{S}_{M}(\varphi_{1})+\mathcal{S}_{M}(\varphi_{2})$ for all $\varphi _{1}, \varphi_{2} \in \mathcal{C}(M,\mathbb{R})$.
\end{description}

So this functional, $\mathcal{S}_{M}$, is a seminorm in $\mathcal{C}(M,\mathbb{R})$. Suppose this seminorm verifies the following invariance property for homeomorphisms:
\begin{description}
\item[iv)] if $N$ is a  topological space, then $\mathcal{S}_{M}(\varphi)=\mathcal{S}_{N}(\varphi \circ h^{-1})$ for all $\varphi  \in \mathcal{C}(M,\mathbb{R})$ and $h\in \mathcal{H}(M,N)$.
\end{description}

\begin{prop}
 The setting
$$
\delta_{exo}((M,\varphi),(N,\psi))=
\begin{cases}
\underset{h\in \Ho(M,N)}{\inf}\mathcal{S}_M(\varphi-\psi\circ h)& \textrm{ if } \mathcal{H}(M,N)\not=\emptyset,\\
\infty & \textrm{ otherwise}.
\end{cases}
$$
defines an extended pseudodistance in \emph{Size} and a pseudodistance in each equivalence class of $Size/\hspace{-0.125cm}\approx$.

\end{prop}

\begin{proof}
 Obviously $\delta_{exo}((M,\varphi),(M,\varphi))=0$.\\
Since  $\SM(\varphi-\psi \circ h)\geqslant 0$ for all $h\in \mathcal{H}(M,N)$, then  $\delta_{exo}((M,\varphi),(N,\psi))\geqslant 0$.\\
The symmetry can be deduced, using \textbf{ii)} and \textbf{iv)}, from the equalities
$$\SM(\varphi-\psi \circ h)=\mathcal{S}_{N}(\varphi\circ h^{-1}-\psi)=\mathcal{S}_{N}(\psi-\varphi\circ h^{-1})$$
which hold for all $h\in \mathcal{H}(M,N)$.

Finally, the triangle inequality
$$
\delta_{exo}((M,\varphi),(T,\xi)) \le \delta_{exo}((M,\varphi),(N,\psi)) + \delta_{exo}((N,\psi),(T,\xi))
$$

\noindent follows from the property $\textbf{iii)}$ and \textbf{iv)} of $\SM$. Indeed, for all $h\in \mathcal{H}(M,N)$ and $g\in \mathcal{H}(N,T)$, we have that:
\begin{align*}
\SM(\varphi-\xi\circ g \circ h)& = \SM(\varphi-\psi\circ h+\psi\circ h-\xi\circ g\circ h)\\
                           & \le \SM(\varphi-\psi\circ h) +\SM(\psi\circ h-\xi\circ g\circ h)\\
                           & = \SM(\varphi-\psi\circ h) +\mathcal{S}_N(\psi-\xi\circ g).
\end{align*}

\noindent To conclude, we note that

$$\inf_{\shortstack{$\scriptstyle h\in \Ho(M,N)$\\
$\scriptstyle g\in \Ho(N,T)$}}\SM(\varphi-\xi\circ g \circ h)= \underset{l\in \Ho(M,T)}{\inf}\SM(\varphi-\xi\circ l)\\
= \delta_{exo}((M,\varphi),(T,\xi))$$

$$
\inf_{\shortstack{$\scriptstyle h\in \Ho(M,N)$\\
$\scriptstyle g\in \Ho(N,T)$}}(\SM(\varphi-\psi\circ h) +\mathcal{S}_N(\psi-\xi\circ g))=\delta_{exo}((M,\varphi),(N,\psi)) + \delta_{exo}((N,\psi),(T,\xi)).
$$


\end{proof}

\begin{defn}
We shall call \emph{exotic pseudodistance} associated with the family of seminorms $\{\SM\}$ the function
$$\delta_{exo}:Size\times Size\longrightarrow \mathbb R\cup \lbrace\infty\rbrace$$
so defined
$$
\delta_{exo}((M,\varphi),(N,\psi))=
\begin{cases}
\underset{h\in \Ho(M,N)}{\inf}\mathcal{S}_M(\varphi-\psi\circ h)& \textrm{ if } \mathcal{H}(M,N)\not=\emptyset,\\
\infty & \textrm{ otherwise}.
\end{cases}
$$
\end{defn}

\vskip.4cm
More precisely, we should denote $\delta_{exo}$ by $\delta_{exo}^{\{\SM\}}$ as this pseudodistance depends on the family of seminors in question, whereby $\{\SM\}$ is omitted for the sake of simplicity.
\vskip.5cm
In the remainder of this paper, if the  functional
$\mathcal{S}_{M}: \mathcal{C}(M,\mathbb{R})\longrightarrow
\mathbb{R}$ is defined by
$\mathcal{S}_{M}(\varphi)=\max\varphi-\min\varphi$,
we shall denote by $\delta_\Lambda$ the exotic pseudodistance associated with this family of seminorms.
\vskip.5cm

\begin{rem}
To compare two size pairs $(M,\varphi),(N,\psi)$ we can choose different reparametrization invariant (semi)norms depending on the sets $M$ and $N$ and the properties we want to emphasize. A first approach has been studied in the one dimensional case in \cite{FrLa2}.
\end{rem}

\section{Lower bound for the exotic pseudodistance $\delta_\Lambda$}

The computation of $\delta_\Lambda$ involves all the homeomorphisms between two topological spaces, so it is difficult to compute. Hence we need to find a way to obtain information on the exotic pseudodistance in order to compare two topological spaces. In the natural case, that is, when $\delta ((M,\varphi),(N,\psi))=\inf _{h\in \Ho(M,N)} \max _{p\in M} \vert\varphi(p)-\nolinebreak \psi\circ h(p)\vert$, we use size functions to obtain information on the natural pseudodistance, see e.g. \cite{DFr1}. In the exotic case we introduce the following construction.

\vskip 0.5cm

 Let $(M, \varphi)$ and $(N,\psi)$ be  two size pairs. Consider the product spaces $M\times M$ and $N\times N$, and  the measuring functions  $\Phi :M\times M \longrightarrow \mathbb R$ and $\Psi :N\times N \longrightarrow \mathbb R$  defined by $\Phi (p,q)=\varphi (p)-\varphi(q)$ and $\Psi (p,q)=\psi (p)-\psi(q)$, respectively. Let $\ell _{(M\times M, \Phi)}:\bigtriangleup ^+\longrightarrow \mathbb N$ and  $\ell _{(N\times N, \Psi)}:\bigtriangleup ^+\longrightarrow \mathbb N $ be the associated \emph{size functions}.

\vskip 0.5cm

The matching distance associated with a size function  provides an easily computable lower bound for the natural pseudodistance (see Section \ref{seccion 2.2}).  In the following we introduce our main result, providing a sharp lower bound for the exotic pseudodistance between two size pairs $(M, \varphi)$ and $(N,\psi)$ in terms of the matching distance of the product spaces $(M\times M, \Phi)$ and $(N\times N, \Psi)$. To do so, we need the next lemma.

\begin{lem}\label{lema1}
Let $(M,\varphi)$ and $(N,\psi)$ be two size pairs and $(M\times M, \Phi)$ and $(N\times N, \Psi)$ be the size pairs of the product spaces. We have the following inequality
$$\delta ((M\times M, \Phi),(N\times N, \Psi))\leq \delta _{\Lambda}((M,\varphi),(N,\psi)).$$
\end{lem}

\begin{proof}
Set $d _{\La}=\delta _{\La}((M,\varphi),(N,\psi))$. The thesis we want to prove is that
\begin{equation}\label{natural}
  \underset{H\in \Ho(M\times M,N\times N)}{\inf}\left(\underset{(p,q)\in M\times M}{\max}|\Phi(p,q)-\Psi\circ H(p,q)|\right)\le d_{\La}.
 \end{equation}

\noindent By definition of $d_{\La}$, for every real number $\varepsilon>0$ there is some $\overline{h}\in \Ho(M,N)$ such that $$d_{\La}\le \underset{p\in M}{\max}(\varphi(p)-\psi(\overline{h}(p)))-\underset{p\in M}{\min}(\varphi(p)-\psi(\overline{h}(p)))\le d_{\La}+\varepsilon.$$

In order to prove our thesis we bound $|\Phi(p,q)-\Psi\circ H(p,q)|$ in the following way. Consider the homeomorphism $\overline{H}\in \Ho(M\times M, N\times N)$ given by $\overline{H}=(\overline{h},\overline{h})$. Then, for all $(p,q)\in M\times M$
 \begin{align*}
 |\Phi(p,q)-\Psi\circ \overline{H}(p,q)|&=|\Phi(p,q)-(\Psi\circ (\overline{h},\overline{h}))(p,q)|\\
                             &=|\varphi(p)-\varphi(q)-(\psi(\overline{h}(p))-\psi(\overline{h}(q)))| \\
                             &=|(\varphi(p)-\psi(\overline{h}(p))-(\varphi(q)-\psi(\overline{h}(q)))|\\
                             &\le \underset{p'\in M}{\max}(\varphi(p')-\psi(\overline{h}(p')))-\underset{p'\in M}{\min}(\varphi(p')-\psi(\overline{h}(p')))\\
                             & \le d_{\La}+\varepsilon
 \end{align*}

\noindent Therefore, for all $\varepsilon >  0$
$$\max _{(p,q)\in M\times M}\vert\Phi(p,q)-\Psi \circ \overline{H}(p,q)\vert\leq d_{\La}+\varepsilon$$
so that
$$\inf_{H\in \Ho(M\times M, N\times N)}\max_{(p,q)\in M\times M}\vert\Phi (p,q)-\Psi \circ H (p,q)\vert\leq d_{\La}+\varepsilon$$

\noindent for all $\varepsilon > 0$, and  (1) is proved.

\end{proof}

\begin{rem}
If we consider the following subset of $\Ho(M\times M, N\times N)$

$$\mathcal{K}=\{(h,h):M\times M\longrightarrow N\times N \mbox{ for all }h\in \Ho(M,N)\}$$

\noindent the proof of Lemma \ref{lema1} shows that

$$\delta _{\La}((M,\varphi),(N,\psi))=\inf_{H\in \mathcal{K}}\max_{(p,q)\in M\times M}\vert\Phi (p,q)-\Psi \circ H (p,q)\vert.$$

\noindent Note that this equality immediately follows from the definition
\begin{equation*}
 \delta _{\La}((M,\varphi),(N,\psi))=\inf_{h\in \Ho(M,N)}(\max_{M}(\varphi-\psi\circ h)-\min_{M}(\varphi-\psi\circ h))
\end{equation*}
and the identity
\begin{equation*}
\inf_{h\in \Ho(M,N)}\max_{M\times M}\vert\varphi(p)-\psi(h(p))-(\varphi(q)-\psi(h(q)))\vert =\inf_{(H)\in \mathcal{K}}\max_{M\times M}\vert\Phi (p,q)-\Psi \circ H (p,q)\vert.
\end{equation*}

\end{rem}

\vskip0.4cm

From Lemma \ref{lema1} we can deduce a lower bound for the exotic pseudodistance in
 terms of the size functions of the size pairs of the product spaces.

\begin{thm} \label{teorema}
Let $(M,\varphi)$ and $(N,\psi)$ be two size pairs such that $\Ho(M,N)\neq \emptyset$. Then
$$d_{match}(\ell _{(M\times M, \Phi)},\ell _{(N\times N, \Psi)})\le \delta _{\La}((M,\varphi),(N,\psi)).$$
\end{thm}
\begin{proof}
By applying Theorem \ref{desigualnatural} to the size pairs $(M\times M, \Phi)$ and $(N\times N, \Psi)$ we obtain
$$d_{match}(\ell _{(M\times M, \Phi)},\ell _{(N\times N, \Psi)})\le\delta ((M\times M, \Phi),(N\times N, \Psi)).$$

\noindent Lemma \ref{lema1} concludes our proof.

\end{proof}

\section{An example} \label{example}
Now we can show that the lower bound given in Theorem \ref{teorema} for $\delta _{\La}$ helps us to calculate this exotic pseudodistance. Moreover, it allows us to prove that this lower bound is sharp.

\begin{figure}[htbp]
 \centering
 \begin{minipage}[c]{.40\textwidth}
   \centering
  \includegraphics[width=4.5cm]{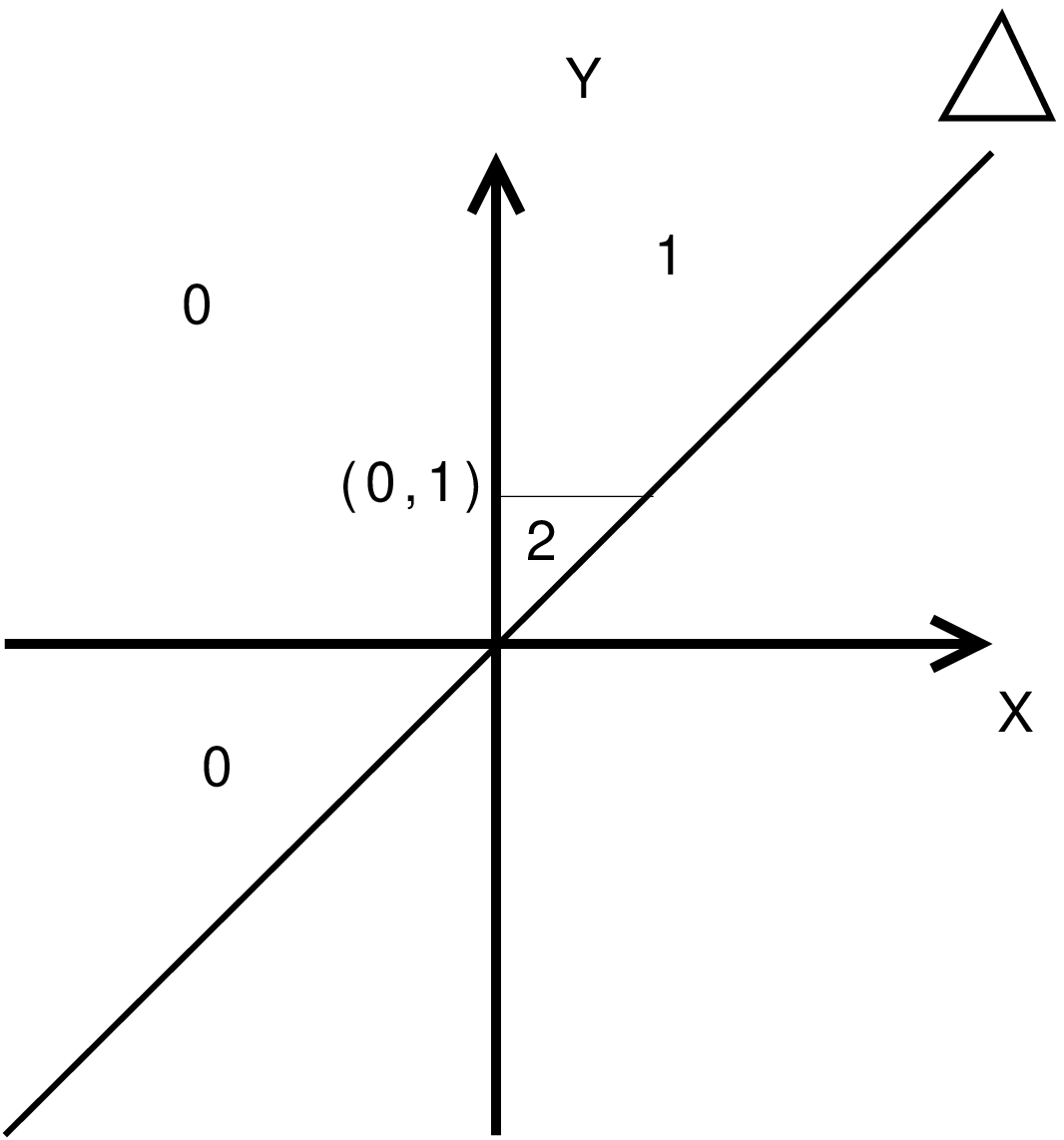}
  \caption{ Size function of $(I,\sin t)$}
\end{minipage}%
\hspace{10mm}%
\begin{minipage}[c]{.40\textwidth}
 \centering
 \includegraphics[width=4.5cm]{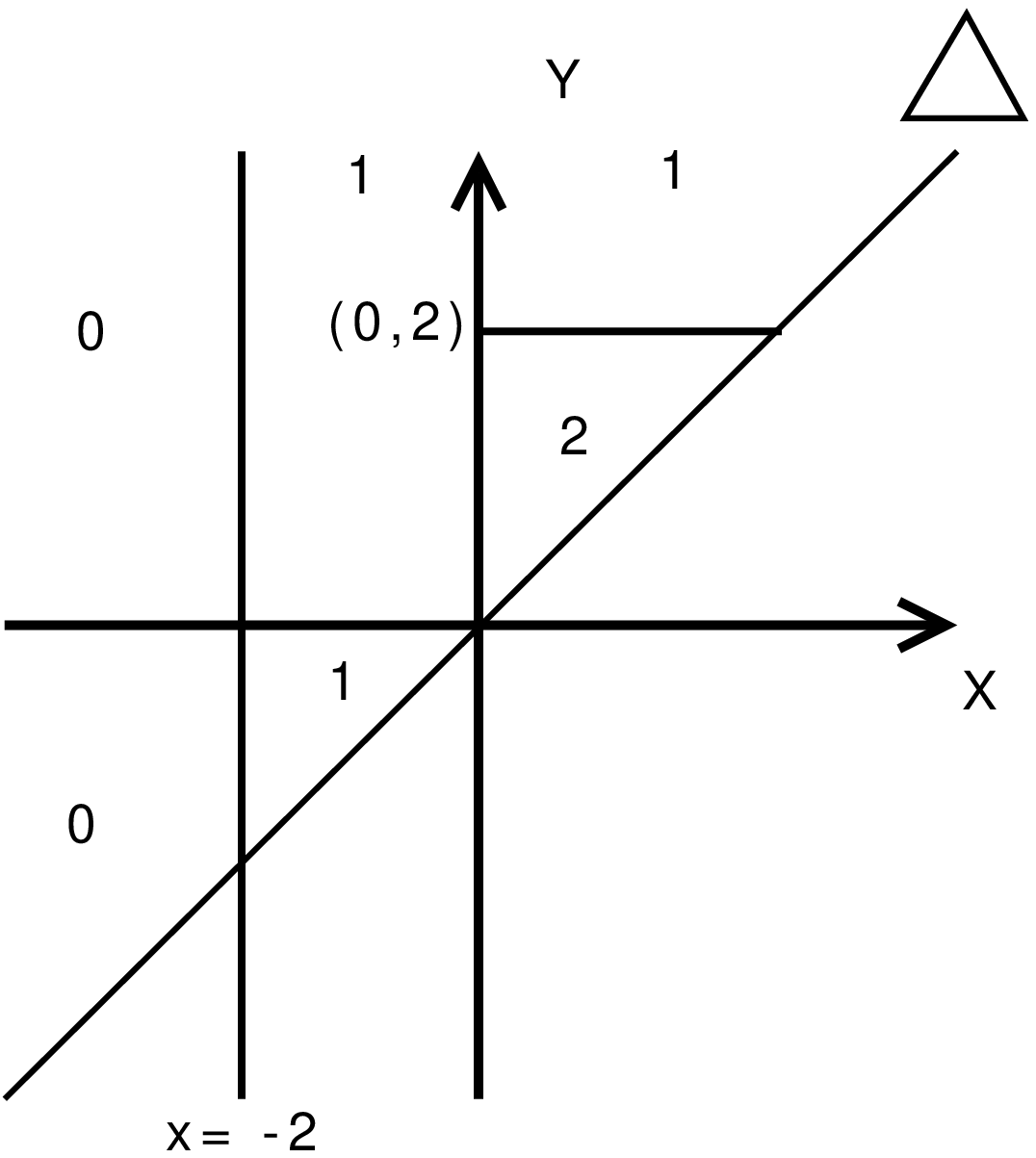}
 \caption{Size function of $(I,2\sin 2t)$}
 \end{minipage}
\end{figure}

\begin{figure}[htbp]
 \centering
  \begin{minipage}[c]{.45\textwidth}
    \rotatebox{270}{\includegraphics[width=9cm]{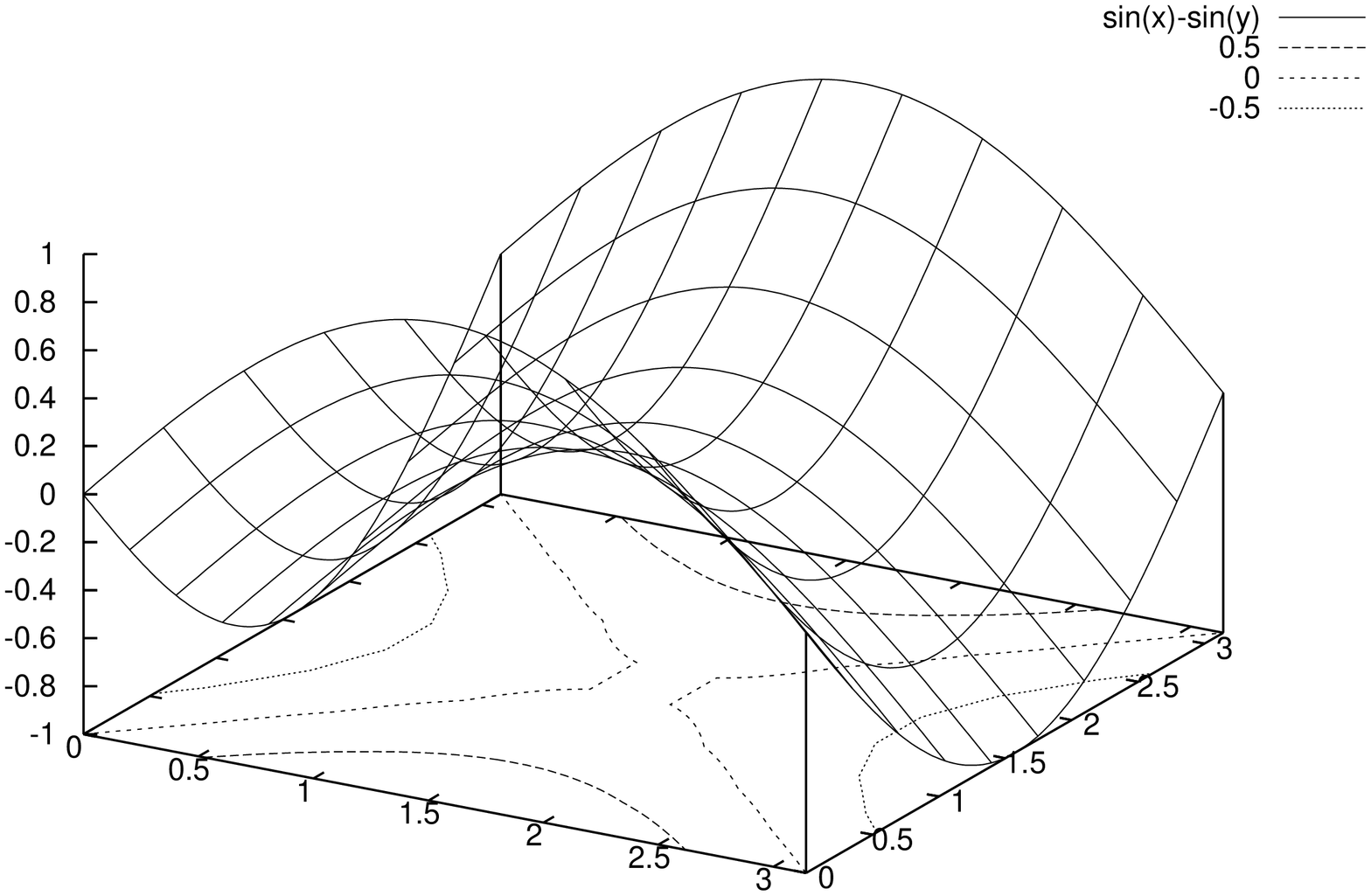}}
    \caption{$\lbrace(t,s,\sin t-\sin s):~s,t\in I\rbrace$}
  \end{minipage}

\centering
  \begin{minipage}[c]{.45\textwidth}
  \rotatebox{270}{\includegraphics[width=9cm]{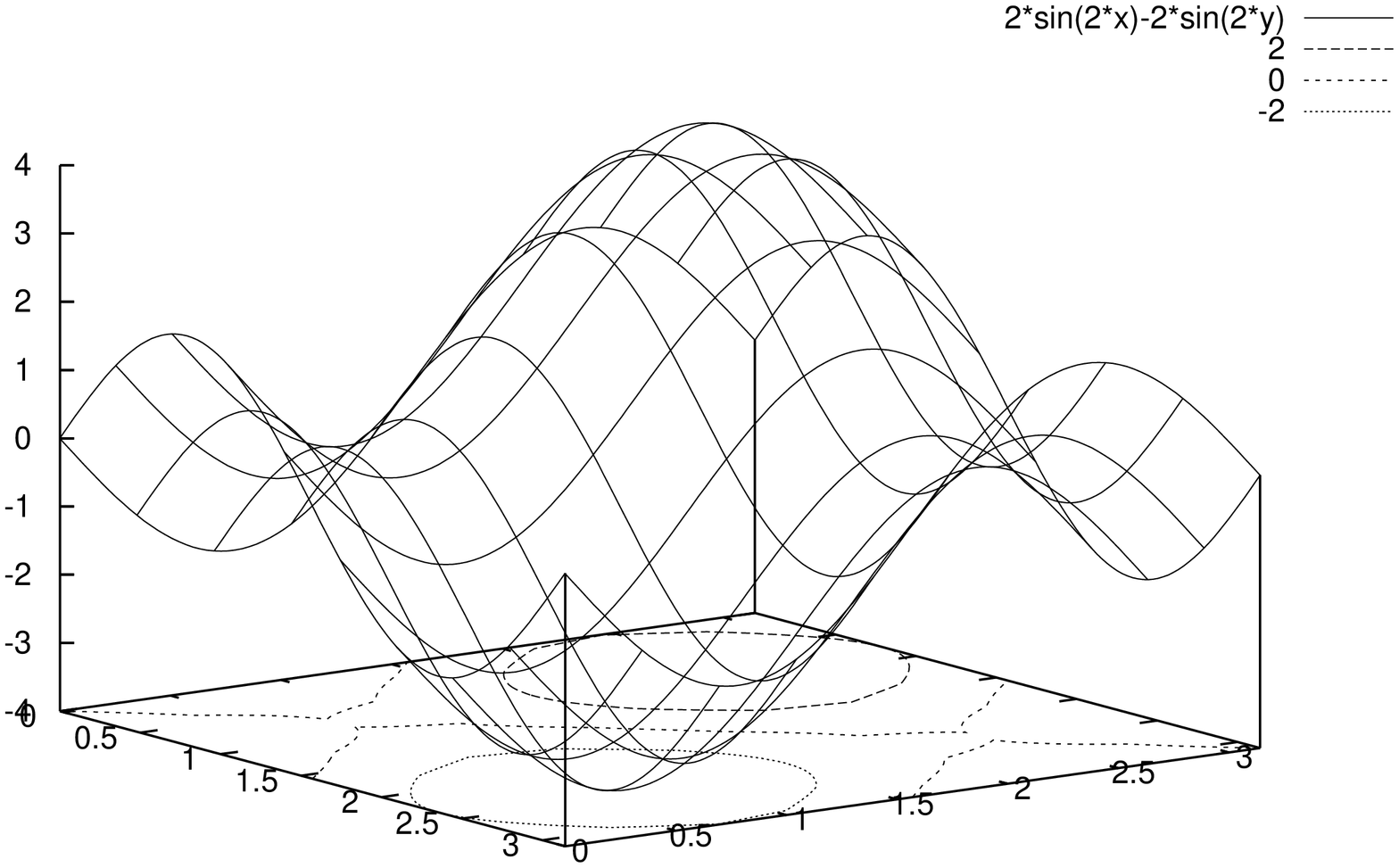}}
  \caption{$\lbrace(t,s,2\sin 2t-2\sin 2s):s,t\in I\rbrace$}
  \end{minipage}
\end{figure}

\noindent Let $(I,\varphi)$ and $(I,\psi)$ be two size pairs where $I=[0,\pi]$, and $\varphi :I\longrightarrow \mathbb{R}$ and $\psi :I\longrightarrow \mathbb{R}$ are defined as $\varphi (t)=\sin t$ and $\psi(t)=2\sin 2t$,  respectively.
\vskip .2cm
Let us calculate first the natural pseudodistance between the size pairs $(I,\sin t)$ and $(I,2\sin 2t)$. In this case we have $d_{match}(\ell_{(I,\varphi)},\ell _{(I,\psi)})=2$ (see Figures 1 and 2) and Theorem \ref{desigualnatural} provides us with a lower bound for the natural pseudodistance
$$2\leq\delta((I, \sin t),(I, 2\sin 2t)).$$

\noindent For every $\varepsilon > 0$, we can easily find an homeomorphism $g:I\longrightarrow I$ such that $\max _{t\in I}\vert \sin t-2\sin 2g(t)\vert\leq 2+\varepsilon$. So the natural pseudodistance between $(I,\sin t)$ and $(I,2\sin 2t)$ is $2$.

In the product space $I\times I$ we construct two measuring  functions as follows:
$$\Phi : I\times I\rightarrow \mathbb{R}\mbox{ defined as } \Phi (t,s)=\sin t-\sin s$$
\noindent and
$$\Psi : I\times I\rightarrow \mathbb{R}\mbox{ defined as } \Psi (t,s)=2\sin 2t-2\sin 2s.$$

Now we consider the size pairs $(I\times I, \Phi)$ and $(I\times I, \Psi)$ and compute the matching distance between the associated size functions. First, we calculate the size functions $\ell _{(I\times I, \Phi)}$ and $\ell _{(I\times I, \Psi)}$ with the help of Figure 3 and Figure 4.

\begin{figure}[htbp]
 \centering
 \begin{minipage}[c]{.45\textwidth}
    {\includegraphics[width=4.5cm]{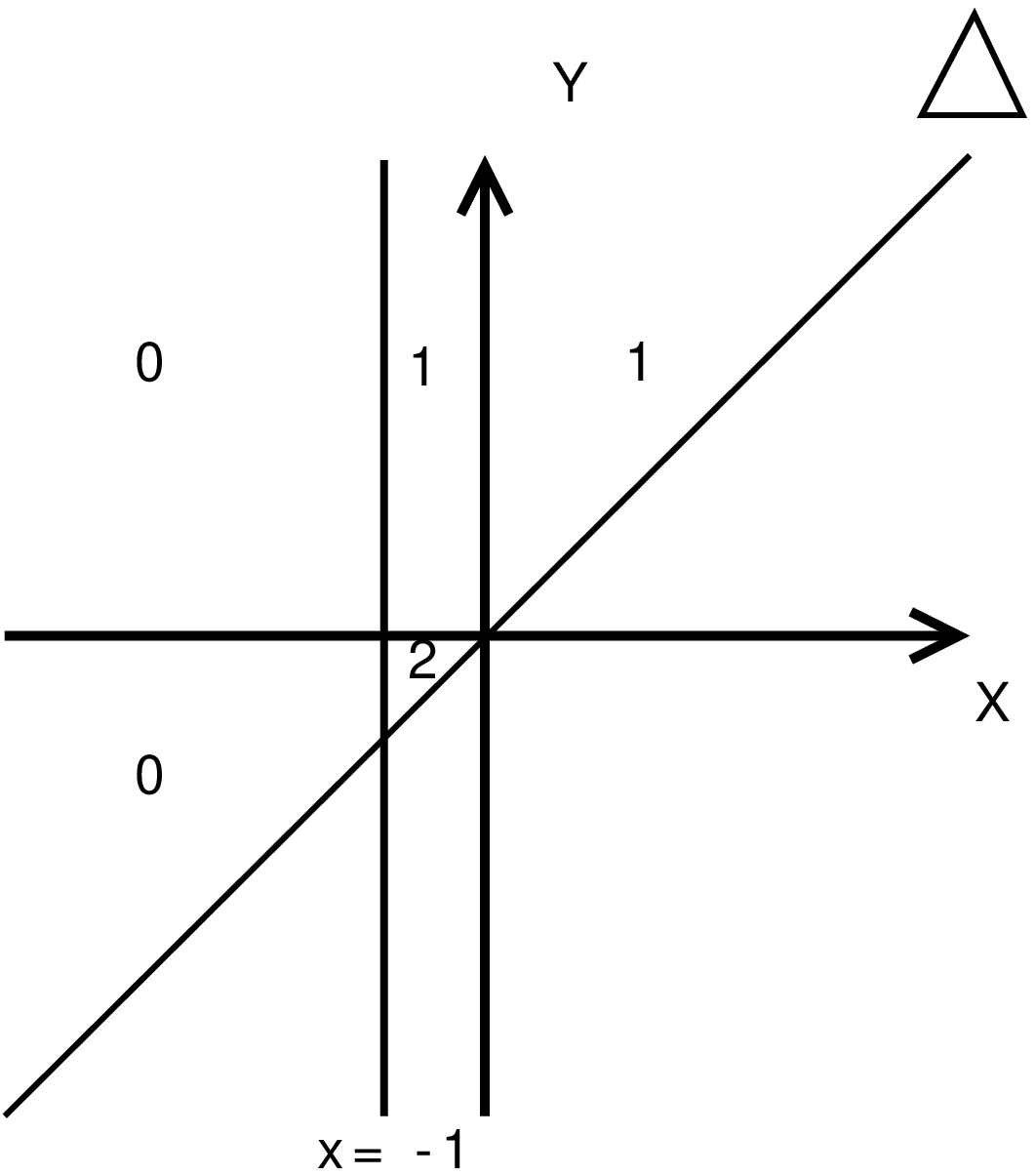}}
    \caption{Size function of $(I\times I,\Phi)$}
  \end{minipage}%
 \hspace{10mm}
  \begin{minipage}[c]{.45\textwidth}
    \centering
    {\includegraphics[width=4.5cm]{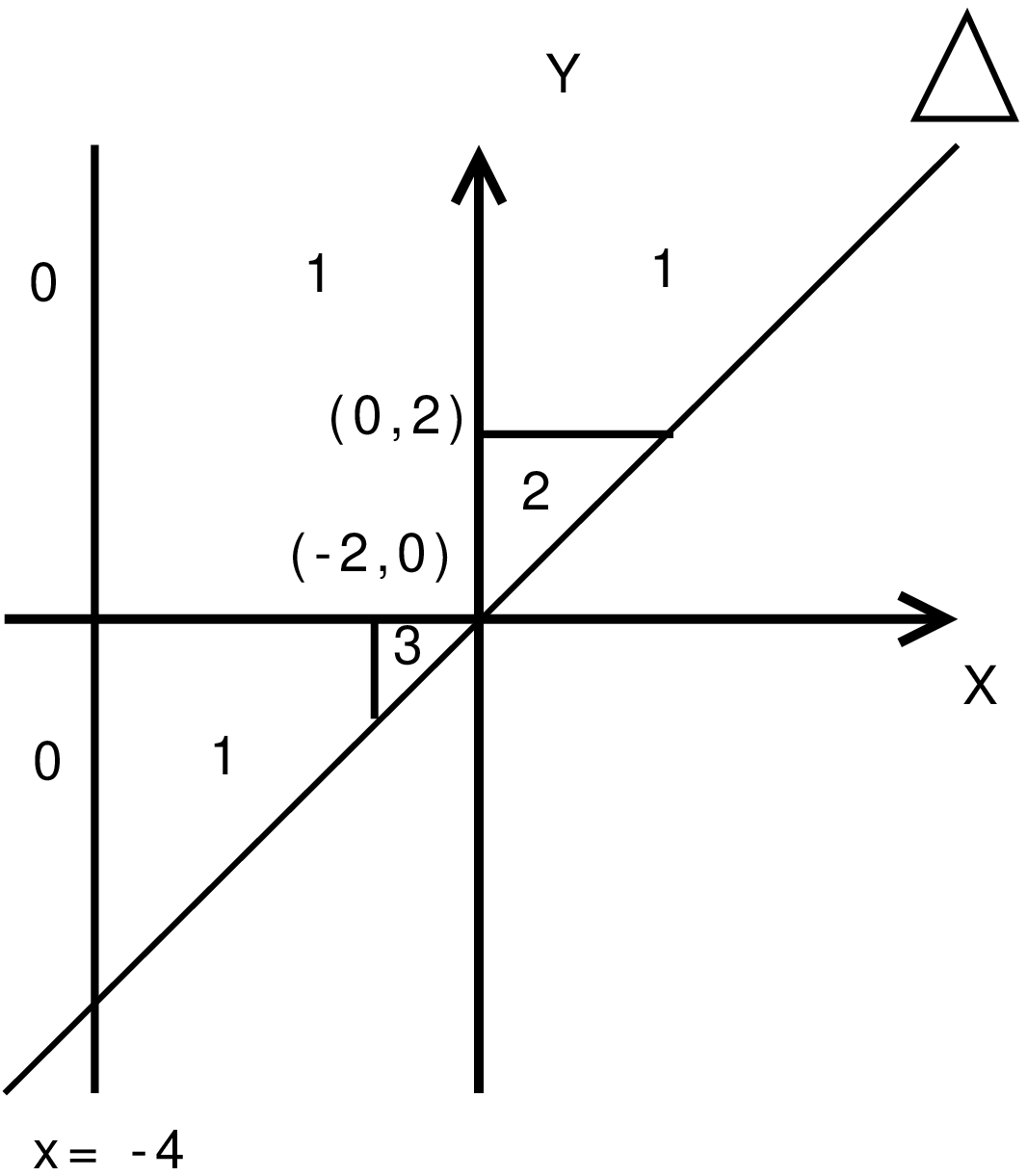}}
    \caption{Size function of $(I\times I,\Psi)$}
  \end{minipage}%
  \hspace{10mm}%
\end{figure}

Set $d_{m}=d_{match}(\ell _{(I\times I, \Phi)},\ell _{(I\times I, \Psi)})$. Figures 5 and 6 show that $d_{m}=3$.
\noindent Applying Theorem \ref{teorema} we have that $d_{m}$ is a lower bound for the exotic pseudodistance, $\delta_\La$, between $(I,\sin t)$ and $(I,2\sin 2t)$, i.e.
 $$3\leq \delta _{\La}((I,\varphi), (I,\psi)).$$

We can easily see that for every $\varepsilon >0$ a homeomorphism $g:I\longrightarrow I$ exists, such that $g$ takes the interval $[0,\pi - \varepsilon)$ to the interval $[0,\frac{\pi}{2})$ and the interval $[\pi - \varepsilon, \pi]$  to the interval  $[\frac{\pi}{2},\pi]$.

\noindent This homeomorphism  verifies the inequality
$$\max_{t\in I} (\sin t-2\sin 2g(t))-\min_{t\in I}(\sin t-2\sin 2g(t))= 3+\eta$$ with $\eta$ a positive real number depending on $\varepsilon$, such that $\lim_{\varepsilon\rightarrow 0}\eta=0$.\\ Hence $\delta _{\La}((I,\sin t), (I,2\sin 2t))\leq 3.$

Therefore by this last inequality and the one given by Theorem \ref{teorema} we have
$$\delta _{\La}((I,\sin t), (I,2\sin 2t))= 3.$$

\noindent This proves that the lower bound for $\delta _{\La}$ given by Theorem \ref{teorema} is sharp.

 \vskip 0.4cm
 $\mathbf{Acknowledgements.}$ I am very grateful to
Patrizio Frosini for his essential help, useful suggestions and
comments on the elaboration of this paper. Thanks also to Claudia Landi for the stimulating discussions.


\begin{thebibliography}{99}


\bibitem[BFlFa]{BiFlFa}{S.~Biasotti, L.~De Floriani, B.~Falcidieno, P.~Frosini, D.~Giorgi, C.~Landi,      L.~Papaleo, M.~Spagnuolo, \emph{Describing shapes by geometrical-topological properties of real functions}, ACM Comput. Surv., vol. 40 (2008), n.2, 161-179.}





\bibitem[dAFrLa]{dAFrLa} {M.~d'Amico, P.~Frosini, C.~Landi, \emph{Natural pseudo-distance and optimal matching between reduced size functions}, Acta Applicandae Mathematicae (to appear).}


\bibitem[DFr1]{DFr1}{P.~Donatini and P.~Frosini, \emph{Lower bounds for natural pseudodistances
via size functions}, Arch. Inequal. Appl. 2 (2004), no. 1, 1--12.}

\bibitem[DFr2]{DFr2}{P.~Donatini and P.~Frosini, \emph{Natural pseudodistances between
closed manifolds}, Forum Mathematicum \textbf{16} (2004), no. 5,
695--715.}

\bibitem[DFr3]{DFr3}
{P.~Donatini and P.~Frosini, \emph{Natural pseudodistances between
closed surfaces},
Journal of the European Mathematical Society, vol. 9 (2007), n. 2, 231-253.}

\bibitem[EH]{EH}{H.~Edelsbrunner and J.~Harer. \emph{Persistent homology --- a survey}. Surveys on discrete and computational geometry,  257--282, Contemp. Math., 453, Amer. Math. Soc., Providence, RI, 2008.}

\bibitem[FeFrUV]{FeFrUV}{M.~Ferri, P.~Frosini, C.~ Uras and A.~Verri  , \emph{On the use of size functions for shape analysis}, Biol. Cybern. 70 (1993), 99-107.}



\bibitem[FeLoP]{FeLoP}{M.~Ferri, S.~Lombardini and C.~Pallotti  , \emph{Leukocyte classification by size functions}, Proceedings of the second IEEE Workshop an Applications of Computer Vision, IEEE Computer Society Press, Los Alamitos, CA (1994), 223-229.}





\bibitem[Fr]{Fr}{P.~Frosini, \emph{Connections between size functions and critical points}, Math. Meth. Appl. Sci., 19 (1996), 555-569.}





\bibitem[FrLa1]{FrLa1}{P.~Frosini, C.~Landi, \emph{Size Functions and Formal Series},
        Appl. Algebra Eng. Commun. Comput. 12, No.4, 327-349 (2001).}


\bibitem[FrLa2]{FrLa2}{P.~Frosini, C.~Landi, \emph{Reparametrization invariant norms},
          Trans. Amer. Math. Soc.  361  (2009),  no. 1, 407--452.}


\bibitem[G]{G}{R.~Ghrist, \emph{Barcodes: the persistent topology of data},
          Bull. Amer. Math. Soc.  45(1), (2008), 61--75.}





\bibitem[UV1]{UV1}{C.~ Uras and A.~Verri , \emph{Computing size functions from edge maps}, J. Comput. Vision, Volume 23, Number 2, June 1997 , pp. 169-183(15).}

\bibitem[UV2]{UV2}{C.~ Uras and A.~Verri , \emph{Metric-topological approach to shape representation and recognition}, Image Vision Comput., 14 (1996), 189-207.}



\end{thebibliography}
\end{document}